\documentclass[12pt,twoside]{article}
\usepackage{latexsym}
\usepackage{amssymb}
\usepackage{amsfonts}
\usepackage{amstext}
\usepackage{multicol}
\usepackage{amsmath, amsthm}
\usepackage{setspace}
\newtheorem{theorem}{Theorem}
\newtheorem{lemma}[theorem]{Lemma}
\newtheorem*{MS}{Theorem A}
\newtheorem*{MSp}{Theorem B}

\begin{document}
\title{Upper bound for the Dvoretzky dimension in Milman-Schechtman theorem}
\date{}
\author{Han Huang, Feng Wei}
\maketitle

\begin{abstract}

For a symmetric convex body $K\subset\mathbb{R}^n$, the Dvoretzky dimension $k(K)$  is the largest dimension for which a random central section of $K$ is almost spherical. A Dvoretzky-type theorem proved by V.~D.~Milman in 1971 provides a lower bound for $k(K)$ in terms of the average $M(K)$ and the maximum $b(K)$ of the norm generated by $K$ over the Euclidean unit sphere.
Later, V.~D.~Milman and G. Schechtman obtained a matching upper bound for $k(K)$ in the case when $\frac{M(K)}{b(K)}>c(\frac{\log(n)}{n})^{\frac{1}{2}}$.
 In this paper, we will give an elementary proof  of the upper bound in Milman-Schechtman theorem which does not require any restriction on $M(K)$ and $b(K)$.

\end{abstract}

\begin{flushleft}

\section{Introduction}
 Given a symmetric convex body $K$ in $\mathbb{R}^n$, we have a corresponding norm  $\|x\|_K=\inf \{ r>0 \, , \, x \in rK\}$. Let $|\cdot|$ denote the Euclidean norm, $\nu_n$ denote the normalized Haar measure on the Euclidean sphere, $S^{n-1}$, and  $\nu_{n,k}$ denote the normalized Haar measure on the Grassmannian manifold $Gr_{n,k}$. Let  $M=M(K):=\int_{S^{n-1}}\|x\|_Kd\nu_n$ and $b=b(K):=\sup\{\|x\|_K\, , \, x\in S^{n-1}\}$ be  the mean and the maximum of the norm over the unit sphere.

 In 1971, V.~D.~Milman proved the following Dvoretzky-type theorem \cite{M}:
  \begin{theorem} \label{DM}
	Let $K$ be a symmetric convex body in $\mathbb{R}^n$. Assume that $\|x\|_K\le b|x|$ for all $x \in \mathbb{R}^n$. For any $\epsilon \in (0,1)$, there is  $k\ge C_\epsilon(M/b)^2n$ such that
	$$\nu_{n,k}\{ F\in G_{n,k}: (1-\varepsilon)M < \| \cdot \|_{K \cap F} < (1+\varepsilon)M \}>1-\exp(-\tilde{c} k)$$
where $\tilde{c}>0$ is a universal constant, $C_\epsilon>0$ is a constant only depending on $\epsilon$.
\end{theorem}
The quantity $C_\epsilon$ was of the order $\epsilon^2\log^{-1}(\frac{1}{\epsilon})$  in the original proof of V.~D.~Milman. It was improved to the order of $\epsilon^2$ by Y. Gordon \cite{YG} and later, with a simpler argument, by G.~Schechtman \cite{GS}.

In 1997, V.~D.~Milman and G.~Schechtman  \cite{MS} found that the bound on $k$ appearing in Theorem \ref{DM} is essentially optimal. More precisely, they proved the following theorem.

  \begin{MS}
  (Milman--Schechtman, see e.g., section 5.3 in \cite{AGA}).
  Let $K$ be a symmetric convex body in $\mathbb{R}^n$. For $\epsilon \in (0,1)$, define $k(K)$ to be the largest dimension $k$ such that
	$$
\nu_{n,k}\left(\{ F\in G_{n,k}: \forall x\in S^{n-1}\cap F \, , \,  (1-\varepsilon)M < \| x \|_{K} < (1+\varepsilon)M \}\right)>p_{n,k}=\frac{n}{n+k}.
	$$
	Then,
	$$
		\tilde{C}_\epsilon n(M/b)^2 \ge k(K)\ge \bar{C}_{\epsilon} n(M/b)^2
	$$
	when $\frac{M}{b}>c(\frac{\log(n)}{n})^{\frac{1}{2}}$ for some universal constant $c$, where $\|\cdot\|_F$ denotes the norm corresponding to the convex body $K\cap F$ in $F$, and  $\tilde{C}_\epsilon, \bar{C}_{\epsilon}>0$ are constants depending only on $\epsilon$.
\end{MS}

Because the Dvoretzky-Milman theorem cannot guarantee the lower bound with small $\frac{M}{b}$ for $p_{n,k}=\frac{n}{n+k}$, the original proof required an assumption that $\frac{M}{b}>c(\frac{\log(n)}{n})^{\frac{1}{2}}$ for some $c$. In \cite[p. 197]{AGA}, S.~Artstein-Avidan, A.~A.~Giannopoulos, and V.~D.~Milman  addressed it as an open question whether one can prove the same result when $p_{n,k}$ is a constant, such as $\frac{1}{2}$. When $p_{n,k}=\frac{1}{2}$, the lower estimate on $k(K)$ is a direct result of Dvoretzky-Milman theorem \cite{M}, but the upper bound was unknown. In this paper, we are going to give upper bound estimate with $p_{n,k}=\frac{1}{2}$, our main result is the following theorem:
\begin{MSp} Let $K$ be a symmetric convex body in $\mathbb{R}^n$. Fix a constant $\epsilon \in (0,1)$, let $k(K)$ be the largest dimension such that
	$$
\nu_{n,k}\{ F\in G_{n,k}: (1-\varepsilon)M < \| \cdot \|_{K \cap F} < (1+\varepsilon)M \}>\frac{1}{2}.
	$$
	Then,
	$$
	   C n(M/b)^2 \ge k(X)\ge \bar{C}_{\epsilon} n(M/b)^2
	$$
	where $C>0$ is a universal constant and $\bar{C}_{\epsilon}>0$ is a constant depending only on $\epsilon$.
\end{MSp}
In the next section, we will provide a proof of Theorem B with no restriction on $\frac{M}{b}$. In fact, from the proof, one can see that $\frac{1}{2}$ can be replaced by any $c \in (0,1)$ or $1-\exp(-\tilde ck)$, which is the probability appearing in Milman-Dvoretzky theorem.

\section{Proof of Theorem B}

Let $P_k$ be the orthogonal projection from $S^{n-1}$ to some fixed $k$-dimensional subspace, and $|\cdot|$ be the Euclidean norm. The upper estimate is related to the distribution of $|P_k(x)|$, where $x$ is uniformly distributed on $S^{n-1}$ .\\

Recall the concentration inequality for Lipschitz functions on the sphere (see, e.g., \cite{D}):
\begin{theorem}[Measure Concentration on $S^{n-1}$] \label{DL}
	 Let $f : S^{n-1} \rightarrow \mathbb{R}$ be a Lipschitz continuous function with Lipschitz constant $b$. Then, for every $t>0$,
	$$
    \nu_n(\{x\in S^{n-1} : | f(x) -\mathbb{E}(f)|\ge bt\})\le 4 \exp(-c_0t^2n)
	$$
	where $c_0>0$ is a universal constant.
\end{theorem}


Theorem \ref{DL} implies the following elementary lemma.
\begin{lemma} \label{Le}
	Fix any $c_1>0$, let $P_k$ be an orthogonal projection from $\mathbb{R}^n$ to some subspace $\mathbb{R}^k$.  If $t>\frac{c_1}{\sqrt{n}}$ and $\nu_n(\{ x \in S^{n-1} \, : \, |P_k(x)|<t\})>\frac{1}{2} $, then $k< c_2t^2n$, where $c_2>0$ is a constant depending only on $c_1$.

\end{lemma}

\begin{proof}[Proof]
 $|P_{k}(x)|$ is a 1-Lipschitz function on $S^{n-1}$ with $\mathbb{E}|P_{k}(x)|$ about $\sqrt{\frac{k}{n}}$. If we want the measure of $\{ x \, :\, |P_k(x)|<t\}$ to be greater than $1/2$, then measure concentration will force $\mathbb{E}|P_{k}|$ to be bounded by the size of $t$, which means $k<c_2t^2n$ for some universal constant $c_2$. Since $t^2n>c_1^2$, we may and shall assume $k$ is bigger than some absolute constant in our proof, then adjust $c_2$.\\

\vspace{0.1in}
To make it precise, we will first give a lower bound on $\mathbb{E}|P_k|$. By Theorem \ref{DL},
	$$		\mathbb{\nu}_n( ||P_k(x)|-\mathbb{E}|P_k(x)||^2>t)\le 4\exp(-c_0tn).
	$$
Thus,
\begin{eqnarray*}
\mathbb{E}|P_k|^2-(\mathbb{E}|P_k|)^2
&=&\mathbb{E}(|P_k|(x)-\mathbb{E}|P_k|)^2\\
&<&\int_0^\infty \mathbb{\nu}_n( ||P_k(x)|-\mathbb{E}|P_k(x)||^2>t)dt\\
&\le& \int_0^\infty 4\exp(-c_0tn)dt=\frac{4}{c_0n}.\\
\end{eqnarray*}
		
With  $\mathbb{E}|P_k|^2=\mathbb{E}\sum_{i=1}^k |x_i|^2=\frac{k}{n}$, we get $\mathbb{E}(|P_k|)>\sqrt{\frac{k}{n}-\frac{4}{c_0n}}$. If we assume that $k>\frac{24}{c_0}$, then we have
$$
	\mathbb{E}(|P_k|)>\sqrt{\frac{1}{2}\frac{k}{n}}.
$$
Assuming $k> 8t^2n$, we have
\begin{eqnarray*}
	\mathbb{E}(|P_k|)-t>  \sqrt{\frac{1}{2}\frac{k}{n}}-t\ge\frac{1}{2}\sqrt{\frac{1}{2}\frac{k}{n}}>0.
\end{eqnarray*}
Applying Theorem \ref{DL} again, we obtain
\begin{eqnarray*}
\mathbb{\nu}_n(|P_k|<t)&<&
\mathbb{\nu}_n\left(\left||P_k|-\mathbb{E}|P_k|\right|>\mathbb{E}(|P_k|)-t\right)\le 4\exp(-c_0(\mathbb{E}(|P_k|)-t)^2n)\\
&\le&4\exp(-c_0(\frac{1}{2}\sqrt{\frac{1}{2}\frac{k}{n}})^2n)\le 4\exp(-\frac{c_0}{8}k)\le 4\exp(-3) <\frac{1}{2},
\end{eqnarray*}
which proves our result by contradiction.
\end{proof}

\vspace{0.2in}

\begin{theorem} \label{Te}
Let $K$ be a convex body with inradius $\frac{1}{b}$. For $\epsilon \in (0,1)$, let $k$ be the largest integer such that
$$
\nu_{n,k}\{ F\in G_{n,k}: (1-\varepsilon)M < \| \cdot \|_{K \cap F} < (1+\varepsilon)M \}>\frac{1}{2}.
$$
Then $ k <Cn( \frac{M}{b} )^2 $ where C is an absolute constant.
\end{theorem}
\begin{proof} We may assume $\|e_1\|_K=b$, then $K \subset S=\{x\in \mathbb{R}^n: |x_1|<\frac{1}{b}\}$, thus $ \|x\|_K \ge \|x\|_S=b|\langle x,e_1\rangle|$.
This implies
\begin{equation}\label{bd}
\begin{array}{rl}
	&\{ V \in G_{n,k} \, : \, \forall x \in V\cap S^{n-1}\, , \,    (1-\epsilon)M<\|x\|_K<(1+\epsilon)M\}\\
	
	\subset&\{ V \in G_{n,k} \, : \, \forall x \in V\cap S^{n-1}\, , \,  \|x\|_S<(1+\epsilon)M\}\\
	=&\{ V \in G_{n,k} \, : \,  \sup_{x \in V\cap S^{n-1}}\langle x,e_1\rangle<(1+\epsilon)\frac{M}{b}\}\\
	=&\{ V \in G_{n,k} \, : \,  |P_V(e_1)|<(1+\epsilon)\frac{M}{b}\}\\
\end{array}
\end{equation}
where $P_V$ is the orthogonal projection from $\mathbb{R}^n$ to $V$. If $V$ is uniformly distributed on $G_{n,k}$ and $x$ is uniformly distributed on $S^{n-1}$, then $|P_{V_0}(x)|$ and $|P_{V}(e_1)|$ are equi-distributed for any fixed $k$-dimensional subspace $V_0$. Therefore,
$$
	\nu_{n,k}(\{ V \in G_{n,k} \, : \,  |P_V(e_1)|<(1+\epsilon)\frac{M}{b}\})=\nu_n(\{ x \in S^{n-1} \, : \, |P_{V_0}(x)|<(1+\epsilon)\frac{M}{b}\}).
$$

As shown in the Remark 5.2.2(iii) of \cite[p. 164]{AGA}, the ratio $\frac{M}{b}$ has a lower bound $\frac{c'}{\sqrt{n}}$. Setting $c_1=c'$ and $t=(1+\epsilon)\frac{M}{b}$, it is easy to see that if
$$
\nu_{n,k}\{ F\in G_{n,k}: (1-\varepsilon)M < \| \cdot \|_{K \cap F} < (1+\varepsilon)M \}>\frac{1}{2},
$$
then $k\le c_1(1+\epsilon)^2\left( \frac{M}{b}\right)^2n<Cn( \frac{M}{b} )^2$ by Lemma \ref{Le} and  (\ref{bd}).

\end{proof}

Now we can prove Theorem B as a corollary of Theorem \ref{Te} and Theorem \ref{DM}:
\begin{proof}[Proof of Theorem B]
Theorem \ref{DM} shows that if $C_\epsilon(M/b)^2n>\frac{\log(2)}{\tilde{c}}$, then there is $k\ge C_\epsilon(M/b)^2n$ such that
	$$
\nu_{n,k}\{ F\in G_{n,k}: (1-\varepsilon)M < \| \cdot \|_{F} < (1+\varepsilon)M \}>1-\exp(-\tilde{c} k)>\frac{1}{2}.
	$$
Otherwise, $(M/b)^2n<\frac{\log(2)}{\tilde{c} C_\epsilon}$. Therefore, $k(K)\ge \min \{\frac{\tilde{c} C_\epsilon}{\log(2)}, C_\epsilon \}(M/b)^2n$. Combining it with Theorem \ref{Te}, we get

	$$
		 C( \frac{M}{b} )^2n	\ge k(K)\ge \min \{\frac{\tilde{c} C_\epsilon}{\log(2)}, C_\epsilon \}(M/b)^2n.
	$$

	\end{proof}

Remark. (1) It is worth noticing that the number $\frac{1}{2}$ plays no special role in our proof. Thus, if we define the Dvoretzky dimension to be the largest dimension such that
$$
\nu_{n,k}\{ F\in G_{n,k}: (1-\varepsilon)M < \| \cdot \|_{K \cap F} < (1+\varepsilon)M \}>c
$$
for some $c \in (0,1)$, then exactly the same proof will work. We will still have $k(K) \sim (\frac{M}{b})^2n$.
Similarly, if we fix $\epsilon$ and replace $\frac{1}{2}$ by $1-\exp(-\tilde ck)$, then the lower bound of $k(K)$ is the one from Theorem \ref{DM}.
For $k$ bigger than some absolute constant, we have $1-\exp(-\tilde c k)>\frac{1}{2}$. Thus, the upper bound is still of order $\left(\frac{M}{b}\right)^2n$. Therefore, we can replace $\frac{1}{2}$ by $1-\exp(-\tilde c k)$ in Theorem A. With this probability choice, it also shows Theorem \ref{DM} provides an optimal $k$ depending on $M,b$.\\
\vspace{0.1in}

(2) Usually, we are only interested in $\epsilon \in (0,1)$. In the lower bound,  $\bar{C}_\epsilon=o_{\epsilon}(1)$. It is a natural question to ask if we could improve the upper bound from a universal constant $C$ to $o_{\epsilon}(1)$. Unfortunately, it is not possible due to the following observation.
Let $K={\rm conv}(B_2^n,Re_1)^{\circ}$. By passing from the intersection on $K$ to the projection of $K^{\circ}$, one can show that $k(K)$ does not exceed the maximum dimension $k$ such that $\nu_n (P_{k}(Rx)<1+\epsilon)>\frac{1}{2}$. Choosing $R=\sqrt \frac{n}{l}$, we get $n(\frac{M}{b})^2\sim l$ and $k(X)\sim l$ by Theorem \ref{DL} and a similar argument to that of Lemma \ref{Le}. This example shows that no matter what $\frac{M}{b}$ is, one can not improve the upper bound in Theorem A from an absolute constant $C$ to $o_{\epsilon}(1)$.

\end{flushleft}

\section{Acknowledgement}
We want to thank our advisor Professor Mark Rudelson for his advise and encouragement on solving this problem. And thank both Professor Mark Rudelson and Professor Vitali Milman for encouraging us to organize our result as this paper.

\end{document}